\documentclass[10pt,reqno]{amsart}

\usepackage[utf8]{inputenc}
\usepackage[T1]{fontenc}
\usepackage{lmodern}
\usepackage{amsmath,amssymb,amsthm}
\usepackage{mathtools}
\usepackage{graphicx}
\usepackage{enumitem}
\usepackage[colorlinks=true,linkcolor=blue,citecolor=blue,urlcolor=blue]{hyperref}
\usepackage[margin=1in]{geometry}
\usepackage{setspace}
\usepackage{biblatex}
\newtheorem{theorem}{Theorem}[section]
\newtheorem{lemma}[theorem]{Lemma}
\newtheorem{proposition}[theorem]{Proposition}
\newtheorem{corollary}[theorem]{Corollary}
\newtheorem{definition}[theorem]{Definition}
\theoremstyle{remark}
\newtheorem{remark}[theorem]{Remark}

\newcommand{\R}{\mathbb{R}}

\newcommand{\Z}{\mathbb{Z}}

\newcommand{\Ric}{\operatorname{Ric}}
\newcommand{\Inj}{\operatorname{Inj}}
\newcommand{\Vol}{\operatorname{Vol}}
\newcommand{\vol}{\operatorname{vol}}

\newcommand{\normw}[1]{\lVert #1 \rVert_{w}}
\newcommand{\norml}[1]{\lVert #1 \rVert_{L^{2}}}

\newcommand{\DD}{\mathcal{D}}
\newcommand{\Rw}{R_{w}}

\newcommand{\Vk}{V_{-\kappa}}

\newcommand{\B}[2]{B(#1,\,#2)}

\begin{document}
\title{\textbf{A Weighted Discretization of Riemannian Manifolds with Lower Ricci Bounds}}
\author{Aditya Tiwari}
\address{Department of Mathematical Sciences, Indian Institute of Science Education
and Research Mohali, Sector 81, SAS Nagar, Punjab, 140306, India.}
\email{adityatiwari@iisermohali.ac.in}

\begin{abstract}
Let $(M,g)$ be a connected compact $n$-dimensional Riemannian manifold
with ${\Ric(M,g)\geq-(n-1)\kappa g}$. We introduce a weighted combinatorial Laplacian on $\varepsilon$-discretizations of $M$ and prove a spectral comparison theorem between the weighted graph Laplacian and the Laplace--Beltrami operator. More precisely, the eigenvalues of the two operators are uniformly comparable with constants depending only on $n,\kappa,\varepsilon$,
independently of the injectivity radius. As an application, we prove spectral stability under measured Gromov--Hausdorff convergence. We also recover the Schoen--Wolpert--Yau inequality using the weighted discretization on families of pinching genus-$2$ hyperbolic surfaces.
\end{abstract}

\maketitle

\section{Introduction}
\label{sec:intro}
The study of the eigenspectrum of the Laplacian is central to spectral geometry. Since explicit computation is typically intractable, a natural approach is to discretize the underlying manifold by a suitable graph. Given a Riemannian manifold $(M,g)$, let $V\subset M$ be a maximal $\varepsilon$-separated set and form the graph $X=(V,E)$ by joining $p,q\in V$ whenever $d(p,q)<3\varepsilon$. One can then compare the eigenvalues $\lambda_k(X)$ of a combinatorial Laplacian on $X$ with the eigenvalues $\lambda_k(M)$ of the Laplacian on $M$, ideally with constants that are uniform over a prescribed geometric class.

Kanai~\cite{Kanai1986,Kanai1985} studied manifolds of \emph{bounded
geometry} ($\Ric\geq-(n-1)\kappa g$ together with injectivity radius $\Inj\geq r_0>0$)
and showed that rough isometries preserve qualitative analytic
properties such as the positivity of the isoperimetric constants and volume growth rates. Under
the same bounded geometry assumptions, Mantuano~\cite{Mantuano2005}
proved the uniform two-sided comparison
\[
  c\,\lambda_k(M)\leq\lambda_k(X)\leq C\,\lambda_k(M)
\]
for all $k<|X|$, where the constants depend only on the geometric
class. Fujiwara~\cite{Fujiwara1995} obtained a comparison without
curvature or injectivity assumptions, but with constants depending on
the individual manifold rather than uniformly on the class. 
Burago, Ivanov, and Kurylev~\cite{BuragoIvanovKurylev2014}
established the convergence $\lambda_k(X)\to\lambda_k(M)$ as
$\varepsilon\to0$ for suitably weighted graphs, uniformly under
two-sided sectional curvature bounds, a lower injectivity radius bound,
and an upper diameter bound; see also Aubry~\cite{Aubry2013}. Maity and
Bhattacharya~\cite{MaityBhattacharya2026} replaced the sectional
curvature bounds by two-sided Ricci bounds, and
Inagaki~\cite{Inagaki2025} treated random samples under a lower Ricci
bound together with lower volume and upper diameter bounds.

However, none of these results gives a uniform two-sided spectral comparison under a Ricci lower bound alone, with constants depending only on the geometric class. The obstruction is the possible collapse of local volume, which prevents the standard combinatorial Laplacian from retaining uniform spectral control.

We replace the standard combinatorial Laplacian with a weighted
version in which vertices carry their local volume. For an
$\varepsilon$-discretization $X=(V,E)$ of $(M,g)$, assign to each
vertex $p$ the weight $w(p)=\Vol(\B{p}{\varepsilon})$ and to each edge
$\{p,q\}\in E$ the weight $\mu(p,q)=\min(w(p),w(q))$. The
\emph{weighted combinatorial Laplacian} is
\[
  (\Delta_w f)(p)
  \;=\;
  \frac{1}{w(p)}\sum_{q\sim p}\mu(p,q)\bigl(f(p)-f(q)\bigr),
\]
with the Dirichlet form $\DD(f,f)=\sum_{\{p,q\}\in E}\mu(p,q)|f(p)-f(q)|^2$
and weighted Rayleigh quotient $\Rw(f)=\DD(f,f)/\normw{f}^2$, where
$\normw{f}^2=\sum_{p\in V}f(p)^2w(p)$. The local-volume weights compensate for the loss of uniform volume control due to collapse. In particular, they allow the discrete Dirichlet form and the discrete $L^2$ norm to remain uniformly comparable to their Riemannian counterparts even when the injectivity radius becomes arbitrarily small. This idea
was identified in Mantuano's thesis~\cite{MantuanoThesis} following a
suggestion of G.~Carron; we develop it systematically.

\begin{theorem}[Main theorem]
\label{thm:main}
Let $n\geq 2$, $\kappa\geq 0$, $\varepsilon>0$. There exist constants
$c_1,c_2>0$ depending only on $n,\kappa,\varepsilon$ such that for
every connected compact $n$-dimensional Riemannian manifold $(M,g)$
with $\Ric(M,g)\geq-(n-1)\kappa g$, and every
$\varepsilon$-discretization $X$,
\[
  c_1\,\lambda_k(X)\;\leq\;\lambda_k(M)\;\leq\;c_2\,\lambda_k(X)
  \qquad\text{for all }k<|X|.
\]
In particular, the constants are independent of $M$, $k$, and the injectivity radius.
\end{theorem}

The proof of the main theorem rests on two analytic facts, both available under
$\Ric\geq-(n-1)\kappa g$ alone. First, the local Poincar\'e inequality of
Saloff-Coste~\cite{SaloffCoste1992}, which holds on every geodesic
ball of radius $r$ with a constant
depending only on $n$, $\kappa$ and $r$. It replaces the chart-based
estimates that rely on an injectivity radius bound in~\cite{Mantuano2005}. 
Second,
the Bishop--Gromov volume comparison theorem, which controls both the
degree of the discretization graph and the comparability of adjacent
vertex weights. Apart from Cheng's eigenvalue
comparison, the proof uses only volume doubling and a
$(2,2)$-Poincar\'e inequality at scale~$\varepsilon$. In future work, we plan to investigate
whether this framework extends beyond the Riemannian setting.

For $x\in M$, $\Inj(x)$ denotes the injectivity radius of $M$ at $x$, and
$\Inj(M)=\inf_{x\in M}\Inj(x)$. When $\Inj(M)\geq r_0>0$, Bishop--Gromov gives two-sided bounds on
the vertex weights and $\Delta_w$ becomes spectrally equivalent to the
unweighted combinatorial Laplacian. Thus Theorem~\ref{thm:main} recovers
Theorem~3.7 from~\cite{Mantuano2005}. When $\Inj(M)$ is small or
approaches zero, the vertex weights vary as the local volume varies and
maintain spectral control even in the collapsing case. In the following
result, we illustrate this on pinching families of genus-$2$ hyperbolic
surfaces and show that the weighted graph captures the
Schoen--Wolpert--Yau dichotomy intrinsically.

\begin{proposition}
\label{prop:intro-swy}
Fix $r_0>0$, and let $\{\Sigma_\ell\}$ be a family of closed
hyperbolic surfaces of genus~$2$, each containing a simple closed
geodesic $\gamma_\ell$ of length $\ell$, such that
\begin{equation}
\label{eq:intro-thick}
\Inj(x)\geq r_0
\qquad\text{for every }x\in
\Sigma_\ell\setminus\mathcal C(\gamma_\ell),
\end{equation}
where $\mathcal C(\gamma_\ell)$ denotes the collar of $\gamma_\ell$
given by the Collar Lemma. Then there exists
$\varepsilon_0=\varepsilon_0(r_0)>0$ such that, for every fixed
$\varepsilon\in(0,\varepsilon_0)$, there are constants
$C_1,C_2,C_3>0$ and $\ell_1>0$, depending only on $r_0$ and $\varepsilon$ and
independent of $\ell$, such that every $\varepsilon$-discretization
$X_\ell$ of $\Sigma_\ell$ with $\ell\leq\ell_1$ satisfies
\[
C_1\,\ell\leq\lambda_1(X_\ell)\leq C_2\,\ell
\quad\text{if $\gamma_\ell$ separates},
\qquad
\lambda_1(X_\ell)\geq C_3
\quad\text{if $\gamma_\ell$ is non-separating}.
\]
\end{proposition}

\begin{remark}
\label{rem:scope}
Hypothesis~\eqref{eq:intro-thick} ensures that the degeneration occurs
only along $\gamma_\ell$ and is essential. If a second simple
closed geodesic, disjoint from
$\gamma_\ell$, is allowed to shrink simultaneously, then
$\lambda_1\to0$ may occur even when $\gamma_\ell$ is non-separating.
\end{remark}

Proposition~\ref{prop:intro-swy} is proved in
\S\ref{sec:examples}. Combined
with Theorem~\ref{thm:main}, it yields the Schoen--Wolpert--Yau
inequality~\cite{SchoenWolpertYau1980} on such families.

\begin{corollary}
\label{cor:intro-swy}
With $r_0,\varepsilon$ and $\ell_1$ as in Proposition~\ref{prop:intro-swy}, there
exist constants $\tilde{C}_1,\tilde{C}_2,\tilde{C}_3>0$, depending only
on $r_0,\varepsilon$, such that, for $\ell\leq\ell_1$,
\[
\tilde{C}_1\,\ell \leq \lambda_1(\Sigma_\ell)
\leq \tilde{C}_2\,\ell
\quad\text{when $\gamma_\ell$ separates},
\qquad
\lambda_1(\Sigma_\ell)\geq \tilde{C}_3
\quad\text{when $\gamma_\ell$ is non-separating.}
\]
\end{corollary}

A direct corollary of Theorem~\ref{thm:main} is a uniform spectral
comparison along a tower $\{M_i\}_{i\geq1}$ of finite-sheeted
Riemannian coverings of $M$, for $\varepsilon<\frac12\Inj(M)$, with constants depending only on
$n,\kappa,\varepsilon$ and independent of $i$. Remark~\ref{rem:brooks} relates this to
Brooks' theorem~\cite{Brooks1986} on the Cheeger constants of the
associated Schreier graphs; see
\S\ref{subsec:tower}.

We also obtain a spectral stability result for same-dimensional
manifolds that are close in the measured Gromov--Hausdorff topology.
This may be viewed as a same-dimensional analogue of Mantuano's
Theorem~5.1~\cite{Mantuano2005}, without an injectivity-radius
assumption. Under a lower Ricci bound, convergence of eigenvalues
along measured Gromov--Hausdorff convergent sequences is due to Cheeger
and Colding~\cite{CheegerColding2000}; see also Fukaya~\cite{Fukaya1987}
and Burago, Ivanov, and Kurylev~\cite{BuragoIvanovKurylev2019}.

\begin{theorem}
\label{thm:mGH}
Let $(M,g_M)$ and $(N,g_N)$ be compact Riemannian $n$-manifolds with
$\Ric\geq-(n-1)\kappa g$, equipped with their volume measures. Fix
$\varepsilon>0$, $v_0>0$, and $\eta\in(0,\varepsilon/2)$, and assume
\begin{enumerate}[label=(\roman*)]
\item $\Vol(\B{p}{\varepsilon})\geq v_0$ for all $p\in M\cup N$
  (uniform non-collapsing at scale $\varepsilon$);
\item $d_{\mathrm{mGH}}(M,N)<\eta$ in the sense of
Definition~\ref{def:mGH}.
\end{enumerate}
Then there exist constants $c,C>0$ depending only on
$n,\kappa,\varepsilon,v_0,\eta$ such that
\[
  c\,\lambda_k(M)\;\leq\;\lambda_k(N)\;\leq\;C\,\lambda_k(M)
  \qquad\text{for all }k<\min(|X_M|,|X_N|),
\]
where $X_M$ and $X_N$ are arbitrary $\varepsilon$-discretizations of
$M$ and $N$, respectively.
\end{theorem}

Our proof of Theorem~\ref{thm:mGH} also shows that the uniform
non-collapsing hypothesis $\Vol(\B{p}{\varepsilon})\geq v_0$ can be
replaced by the strictly weaker assumption that the local volumes of
$M$ and $N$ are \emph{mutually comparable} at scale~$\varepsilon$; see
Corollary~\ref{cor:relative-noncollapsing}. 

\noindent\textbf{Organisation.}
In Section~\ref{sec:setup}, we define weighted discretizations and establish the Bishop--Gromov degree and weight estimates. Section~\ref{sec:operators} constructs the smoothing and discretization operators. In Section~\ref{sec:main}, we prove Theorem~\ref{thm:main}, followed by the tower-of-coverings theorem. Section~\ref{sec:applications} proves Theorem~\ref{thm:mGH} and Corollary~\ref{cor:relative-noncollapsing} using a weighted rough-isometry theorem for graphs. Finally, Section~\ref{sec:examples} applies the framework to pinching families of hyperbolic surfaces, recovering the Schoen--Wolpert--Yau inequality.

\textbf{Acknowledgements}.
The author thanks Prof.~Bruno Colbois for suggestions during and
after the workshop \emph{GSTW02: Geometry of Eigenvalues} that
shaped this paper. The author also thanks the Isaac Newton Institute
for Mathematical Sciences, Cambridge, for support and hospitality
during the programme \emph{Geometric Spectral Theory and
Applications}, where work on this paper was initiated. This work was
supported by EPSRC grant no.\ EP/Z000580/1.
\section{Weighted Discretizations}
\label{sec:setup}

Throughout, $(M,g)$ denotes a connected compact Riemannian $n$-manifold,
$n\geq 2$, with
\[\Ric(M,g)\;\geq\;-(n-1)\kappa\,g,
  \qquad\kappa\geq 0.\]
Fix $\varepsilon>0$. We write $\B{p}{r}$ for the open geodesic ball of radius $r$ centred at
$p\in M$, $d(\cdot,\cdot)$ for the Riemannian distance, and $dV$ for the
Riemannian volume form.
The space-form volume function $\Vk(r)$ denotes the volume of a geodesic
ball of radius $r$ in the simply connected $n$-dimensional space form of
constant sectional curvature $-\kappa$.

\begin{definition}
\label{def:disc}
An \emph{$\varepsilon$-discretization} of $M$ is a finite graph $X=(V,E)$
where $V\subset M$ is a maximal $\varepsilon$-separated subset
(i.e.\ $d(v,v')\geq\varepsilon$ for all distinct $v,v'\in V$, and
$\bigcup_{v\in V}\B{v}{\varepsilon}=M$) and
\[
  \{p,q\}\in E
  \;\iff\;
  0<d(p,q)<3\varepsilon.
\]
We write $|X|$ for the cardinality of $V$.
\end{definition}

\begin{definition}
\label{def:weights}
For an $\varepsilon$-discretization $X=(V,E)$ of $M$, define
\[w(p) \;=\; \Vol\bigl(\B{p}{\varepsilon}\bigr)\;
  \;\text{for }\; p\in V,\quad\text{ and }\;\;
  \mu(p,q) \;=\; \min\bigl(w(p),w(q)\bigr)\;\text{ for }\; \{p,q\}\in E.\]
The associated \emph{weighted inner product} on $\ell^2(V,w)$ is
$\langle f,g\rangle_w = \sum_{p\in V}f(p)g(p)w(p)$,
and the \emph{weighted Dirichlet form} is
\[
  \DD(f,f)
  \;=\;
  \sum_{\{p,q\}\in E}|f(p)-f(q)|^2\,\mu(p,q).
\]
The \emph{weighted combinatorial Laplacian} $\Delta_w$ is the unique
self-adjoint operator on $\ell^2(V,w)$ satisfying
$\langle\Delta_w f,g\rangle_w = \DD(f,g)$ for all $f,g:V\to\R$.
Explicitly,
\[
  (\Delta_w f)(p)
  \;=\;
  \frac{1}{w(p)}\sum_{q\sim p}\mu(p,q)\bigl(f(p)-f(q)\bigr).
\]
Its spectrum $0=\lambda_0(X)\leq\lambda_1(X)\leq\cdots\leq\lambda_{|X|-1}(X)$
satisfies
\[
  \lambda_k(X)
  \;=\;
  \inf_{W\in\mathcal{E}_{k+1}}
  \sup_{f\in W\setminus\{0\}}\Rw(f),
  \qquad
  \Rw(f)=\frac{\DD(f,f)}{\normw{f}^2},
\]
where $\mathcal{E}_{k+1}$ denotes the collection of all $(k+1)$-dimensional
subspaces of $\ell^2(V,w)$.
\end{definition}

\begin{remark}
\label{rem:quadform}
By direct computation, $\langle\Delta_w f,f\rangle_w = \DD(f,f)$ for all
$f:V\to\R$. Consequently, if $f$ belongs to the span of the first $k+1$
eigenfunctions of $\Delta_w$, then
$\DD(f,f)\leq\lambda_k(X)\normw{f}^2$.
\end{remark}

In the following lemma, we use the Bishop--Gromov volume comparison theorem to establish estimates required later in the proof. 

\begin{lemma}
\label{lem:BG}
Every $\varepsilon$-discretization $X=(V,E)$ of $M$ satisfies the following.
\begin{enumerate}[label=(\roman*)]
\item\label{it:count}
For every $x\in M$ and $r>0$, the ball $\B{x}{r}$ contains at most
\[
  N_r \;=\; \left[\frac{\Vk(r+\varepsilon/2)}{\Vk(\varepsilon/2)}\right]^2
\]
 points of $V$. In particular, every $v\in V$ has at most $N=N_{3\varepsilon}$ neighbours,
and the covers $\{\B{v}{2\varepsilon}\}_{v\in V}$ and
$\{\B{v}{4\varepsilon}\}_{v\in V}$ have multiplicity at most $N'=N_{2\varepsilon}$ and $N''=N_{4\varepsilon}$ respectively.

\item\label{it:comp}
For every $\{p,q\}\in E$, with $C_d=\Vk(4\varepsilon)/\Vk(\varepsilon)$,
\[
C_d^{-1}\;\leq\;\frac{w(p)}{w(q)}\;\leq\;C_d,
\qquad\text{and hence}\qquad
C_d^{-1}\,w(p)\;\leq\;\mu(p,q)\;\leq\;w(p).
\]

\item\label{it:Rw}
For every nonzero $f:V\to\mathbb R$, $\Rw(f)\leq 2N.$
\end{enumerate}
Here $N,N',N'',C_d$ depend only on $n,\kappa,\varepsilon$.
\end{lemma}

\begin{proof}
\ref{it:count}.
Let $u_1,\ldots,u_m$ be the points of $V$ in $\B{x}{r}$. The
$\varepsilon$-separation of $V$ gives pairwise disjoint balls
$\B{u_i}{\varepsilon/2}\subset\B{x}{r+\varepsilon/2}$, while
$\B{x}{\varepsilon/2}\subset\B{u_i}{r+\varepsilon/2}$. Hence,
by Bishop--Gromov at $u_i$,
\[
\Vol\!\bigl(\B{u_i}{\varepsilon/2}\bigr)
\geq
\Vol\!\bigl(\B{x}{\varepsilon/2}\bigr)
\frac{\Vk(\varepsilon/2)}{\Vk(r+\varepsilon/2)}.
\]
Summing over $i$ and applying Bishop--Gromov at $x$ gives
\[
m\,\Vol\!\bigl(\B{x}{\varepsilon/2}\bigr)
\frac{\Vk(\varepsilon/2)}{\Vk(r+\varepsilon/2)}
\leq
\Vol\!\bigl(\B{x}{r+\varepsilon/2}\bigr)
\leq
\Vol\!\bigl(\B{x}{\varepsilon/2}\bigr)
\frac{\Vk(r+\varepsilon/2)}{\Vk(\varepsilon/2)}.
\]
Thus $m \le N_r$. The neighbours of $v$ lie in $\B{v}{3\varepsilon}$, and
the multiplicity of $\{\B{v}{s}\}_{v\in V}$ at $x$ is the number of points
of $V$ in $\B{x}{s}$.

\noindent \ref{it:comp}. Since $d(p,q)<3\varepsilon$, we have
$\B{p}{\varepsilon}\subset\B{q}{4\varepsilon}$.
Bishop--Gromov at $q$ gives
$\Vol(\B{q}{4\varepsilon})\leq C_d\, w(q)$,
so $w(p)\leq C_d\, w(q)$.
By symmetry $w(q)\leq C_d\, w(p)$, giving the comparability bounds.
Since $\mu(p,q)=\min(w(p),w(q))\geq w(p)/C_d$ and
$\mu(p,q)\leq w(p)$, the bounds on $\mu$ follow.

\noindent \ref{it:Rw}.
Using $|f(p)-f(q)|^2\leq 2(f(p)^2+f(q)^2)$, we obtain
\[
  \DD(f,f)
  \;\leq\;
  2\sum_p f(p)^2\sum_{q\sim p}\mu(p,q)
  \;\leq\;
  2\sum_p f(p)^2\cdot N\cdot w(p)
  \;=\;
  2N\normw{f}^2,
\]
where we used $\mu(p,q)\leq w(p)$ for all $q\sim p$.
\end{proof}

\section{Smoothing and Discretization Operators}
\label{sec:operators}

We adapt the smoothing and discretization operators of
Chavel~\cite{Chavel2001} to the weighted $\ell^2$ setting. Since
Saloff-Coste's local Poincar\'e inequality holds without an injectivity
radius bound, we use it in place of the chart-based gradient estimates
of the bounded-geometry framework.

By Saloff-Coste~\cite{SaloffCoste1992}, under the sole assumption
$\Ric(M,g)\geq-(n-1)\kappa g$, there is a constant $C_P$ depending only
on $n$, $\kappa$ and $\varepsilon$ such that
\begin{equation}
\label{eq:Poincare}
  \int_{\B{p}{r}}\!\bigl(F-\bar{F}_{p,r}\bigr)^2 dV
  \;\leq\;
  C_P\,r^2\!\int_{\B{p}{r}}\!|dF|^2\,dV
\end{equation}
for all $F\in W^{1,2}(M)$, all $p\in M$ and all $0<r\leq4\varepsilon$,
where $\bar{F}_{p,r}=\Vol(\B{p}{r})^{-1}\int_{\B{p}{r}}F\,dV$. Indeed, Saloff-Coste's inequality holds on
balls of radius $r$ with constant $C(n)e^{C(n)\sqrt{\kappa}\,r}$, so one may
take $C_P=C(n)e^{4C(n)\sqrt{\kappa}\,\varepsilon}$.

\subsection{The smoothing operator}

\begin{definition}
\label{def:smooth}
Let $\{\varphi_v\}_{v\in V}$ be a smooth partition of unity subordinate
to the cover $\{\B{v}{2\varepsilon}\}_{v\in V}$ of $M$, with
$|\nabla\varphi_v|\leq B_1/\varepsilon$ pointwise, where $B_1$ is the
constant of Remark~\ref{rem:gradphi}.
The \emph{smoothing operator} $S:\ell^2(V,w)\to C^\infty(M)$ is
\[
  (Sf)(x) \;=\; \sum_{v\in V}\varphi_v(x)\,f(v).
\]
\end{definition}

\begin{remark}
\label{rem:gradphi}
Such a partition of unity exists. Fix once and for all
$\phi\in C^\infty(\R,[0,1])$ with $\phi\equiv1$ on $[0,1]$ and
$\phi\equiv0$ on $[2,\infty)$. Set $\psi_v(x)=\phi(d(x,v)/\varepsilon)$.
Since $d(\cdot,v)$ is Lipschitz-$1$,
$|\nabla\psi_v|\leq\|\phi'\|_\infty/\varepsilon$ a.e.\
The cover $\{\B{v}{2\varepsilon}\}$ has multiplicity at most $N'\leq N$
(Lemma~\ref{lem:BG}\ref{it:count}), so
$\sum_u\psi_u\geq\phi(1)>0$ uniformly, and hence
$\varphi_v=\psi_v/\sum_u\psi_u$ satisfies
$|\nabla\varphi_v|\leq B_1/\varepsilon$ with
$B_1=2N\|\phi'\|_\infty/\phi(1)$, which depends only on $n,\kappa,\varepsilon$.
The bound is independent of the injectivity radius.
\end{remark}

\begin{lemma}
\label{lem:smooth}
Let $B_2=2NB_1^2C_d$. For all $f:V\to\R$,
\begin{align}
  \norml{Sf}^2 &\;\leq\; C_d\,\normw{f}^2,
  \label{eq:L2S}\\
  \norml{d(Sf)}^2 &\;\leq\; B_2\,\varepsilon^{-2}\,\DD(f,f).
  \label{eq:GradS}
\end{align}
\end{lemma}

\begin{proof}
Since $\varphi_v\geq 0$ and $\sum_v\varphi_v=1$, the Cauchy--Schwarz
inequality gives $|(Sf)(x)|^2\leq\sum_v\varphi_v(x)f(v)^2$.
Integrating and using
$\int_M\varphi_v\,dV\leq\Vol(\B{v}{2\varepsilon})\leq C_d\,w(v)$, we obtain
\[
  \norml{Sf}^2
  \;\leq\;
  \sum_v f(v)^2\,\Vol\!\bigl(\B{v}{2\varepsilon}\bigr)
  \;\leq\;
  C_d\sum_v f(v)^2 w(v)
  \;=\;
  C_d\normw{f}^2.
\]

For \eqref{eq:GradS}, differentiating $\sum_v\varphi_v\equiv 1$ gives
$\sum_v\nabla\varphi_v=0$ everywhere.
Hence, for any $p_x\in V$ with $x\in\B{p_x}{\varepsilon}$, we have
\[
  \nabla(Sf)(x)
  \;=\;
  \sum_v\bigl(f(v)-f(p_x)\bigr)\nabla\varphi_v(x).
\]
The only active terms satisfy $d(v,p_x)<3\varepsilon$, so either
$v=p_x$ (in which case the term vanishes) or $v\sim p_x$.
Therefore, applying Cauchy--Schwarz with $|\nabla\varphi_v|\leq B_1/\varepsilon$, we get
\[
  |\nabla(Sf)(x)|^2
  \;\leq\;
  \frac{NB_1^2}{\varepsilon^2}
  \sum_{v\sim p_x}\!\bigl(f(v)-f(p_x)\bigr)^2.
\]
Since $\mu(v,p_x)\geq w(p_x)/C_d$ (Lemma~\ref{lem:BG}\ref{it:comp}),
\[
  |\nabla(Sf)(x)|^2
  \;\leq\;
  \frac{NB_1^2 C_d}{\varepsilon^2\,w(p_x)}
  \sum_{v\sim p_x}\mu(v,p_x)\bigl(f(v)-f(p_x)\bigr)^2.
\]
Choosing $p_x$ measurably, the sets $\{x:p_x=p\}$ partition $M$ with
$\Vol(\{p_x=p\})\leq w(p)$. Hence,
integrating and summing over $p$, we obtain
\[
  \norml{d(Sf)}^2
  \;\leq\;
  \frac{NB_1^2 C_d}{\varepsilon^2}
  \sum_p\sum_{v\sim p}\mu(v,p)\bigl(f(v)-f(p)\bigr)^2
  \;=\;
  \frac{B_2}{\varepsilon^2}\,\DD(f,f).
  \qedhere
\]
\end{proof}

\subsection{The discretization operator}

\begin{definition}
\label{def:disc-op}
The \emph{discretization operator} $D:L^2(M)\to\ell^2(V,w)$ is
\[
  (DF)(p)
  \;=\;
  \frac{1}{w(p)}\int_{\B{p}{\varepsilon}}F(x)\,dV(x).
\]
\end{definition}

\begin{lemma}
\label{lem:disc}
Let $B_3=64\,C_P\,N\,N''$. For all $F\in W^{1,2}(M)$,
\begin{align}
  \normw{DF}^2 &\;\leq\; N\,\norml{F}^2,
  \label{eq:L2D}\\
  \DD(DF,DF) &\;\leq\; B_3\,\varepsilon^{2}\,\norml{dF}^2.
  \label{eq:GradD}
\end{align}
\end{lemma}

\begin{proof}
Since $w(p)^{-1}dV$ restricted to $\B{p}{\varepsilon}$ is a probability
measure, Jensen's inequality gives
$|(DF)(p)|^2\leq w(p)^{-1}\int_{\B{p}{\varepsilon}}F^2\,dV$, so
\[
  \normw{DF}^2\;\leq\;\sum_p\int_{\B{p}{\varepsilon}}F^2\,dV\;\leq\; N\norml{F}^2,
\]
since the cover $\{\B{p}{\varepsilon}\}$ has multiplicity at most
$N_\varepsilon\leq N$ (Lemma~\ref{lem:BG}\ref{it:count}).

Now fix an edge $\{p,q\}\in E$, so $d(p,q)<3\varepsilon$ and
$\B{p}{\varepsilon}\cup\B{q}{\varepsilon}\subset\B{p}{4\varepsilon}$.
Let $\bar{F}_{4,p}$ denote the average of $F$ over $\B{p}{4\varepsilon}$.
By Jensen's inequality,
\begin{align*}
  |(DF)(p)-\bar{F}_{4,p}|^2
  &\leq
  \frac{1}{w(p)}\int_{\B{p}{4\varepsilon}}\!\bigl(F-\bar{F}_{4,p}\bigr)^2 dV,\\
  |(DF)(q)-\bar{F}_{4,p}|^2
  &\leq
  \frac{1}{w(q)}\int_{\B{p}{4\varepsilon}}\!\bigl(F-\bar{F}_{4,p}\bigr)^2 dV.
\end{align*}
By the triangle inequality and setting
$\mathcal{I}_p=\int_{\B{p}{4\varepsilon}}(F-\bar{F}_{4,p})^2\,dV$,
\[
  |(DF)(p)-(DF)(q)|^2
  \;\leq\;
  2\Bigl(\tfrac{1}{w(p)}+\tfrac{1}{w(q)}\Bigr)\mathcal{I}_p.
\]
Multiplying by $\mu(p,q)=\min(w(p),w(q))$ and noting
$\min(w(p),w(q))(w(p)^{-1}+w(q)^{-1})\leq 2$,
\[
  |(DF)(p)-(DF)(q)|^2\mu(p,q)
  \;\leq\;
  4\,\mathcal{I}_p.
\]
By the Poincar\'e inequality~\eqref{eq:Poincare} on $\B{p}{4\varepsilon}$,
$\mathcal{I}_p\leq 16\,C_P\varepsilon^2\int_{\B{p}{4\varepsilon}}|dF|^2\,dV$.
Summing over edges, grouping by $p$ (at most $N$ edges each), and using
the multiplicity bound $N''$ for the cover $\{\B{p}{4\varepsilon}\}$, we obtain
\[
  \DD(DF,DF)
  \;\leq\;
  64\,C_P\,N\,N''\,\varepsilon^2\,\norml{dF}^2
  \;=\;
  B_3\,\varepsilon^2\,\norml{dF}^2.
  \qedhere
\]
\end{proof}

\begin{lemma}[$X\to M\to X$]
\label{lem:compXMX}
For all $f:V\to\R$,
\[
  \normw{f-DSf}^2 \;\leq\; 2C_d^2\,\DD(f,f).
\]
\end{lemma}

\begin{proof}
Define $\theta_v(p)=w(p)^{-1}\int_{\B{p}{\varepsilon}}\varphi_v\,dV\geq 0$;
these satisfy $\sum_v\theta_v(p)=1$ and $\theta_v(p)=0$ unless $v=p$ or
$v\sim p$. Then
\[
  (DSf)(p)-f(p)
  \;=\;
  \sum_v\theta_v(p)\bigl(f(v)-f(p)\bigr),
\]
where the $v=p$ term vanishes. Treating
$\theta_\cdot(p)$ as a probability measure on $\{p\}\cup\{v:v\sim p\}$ and
applying Jensen's inequality,
\[
  |(DSf)(p)-f(p)|^2
  \;\leq\;
  \sum_{v\sim p}\theta_v(p)\bigl(f(v)-f(p)\bigr)^2.
\]
Since
$\theta_v(p)\, w(p)
 =\int_{\B{p}{\varepsilon}}\varphi_v\,dV
 \leq\Vol(\B{v}{2\varepsilon})
 \leq C_d\,w(v)$
and $w(v)\leq C_d\,\mu(v,p)$, we get
\[
  |(DSf)(p)-f(p)|^2\, w(p)
  \;\leq\;
  C_d^2\sum_{v\sim p}\mu(v,p)\bigl(f(v)-f(p)\bigr)^2.
\]
The result follows by summing over $p$ and counting each edge twice.
\end{proof}

\begin{lemma}[$M\to X\to M$]
\label{lem:compMXM}
Let $B_4=8(1+C_d)\,C_P\,N'$. For all $F\in W^{1,2}(M)$,
\[
  \norml{F-SDF}^2 \;\leq\; B_4\,\varepsilon^{2}\,\norml{dF}^2.
\]
\end{lemma}

\begin{proof}
Since $\sum_p\varphi_p\equiv 1$,
$F(x)-(SDF)(x)=\sum_p\varphi_p(x)(F(x)-(DF)(p))$.
By Jensen's inequality,
\[
  \norml{F-SDF}^2
  \;\leq\;
  \sum_p\int_{\B{p}{2\varepsilon}}\bigl(F-(DF)(p)\bigr)^2 dV.
\]
Fix $p$ and let $\bar{F}_{p,2}$ denote the average of $F$ over
$\B{p}{2\varepsilon}$.
Then, by the triangle inequality,
\[
  \int_{\B{p}{2\varepsilon}}\!\bigl(F-(DF)(p)\bigr)^2 dV
  \;\leq\;
  2\int_{\B{p}{2\varepsilon}}\!\bigl(F-\bar{F}_{p,2}\bigr)^2 dV
  \;+\;
  2\,\Vol\!\bigl(\B{p}{2\varepsilon}\bigr)\cdot|\bar{F}_{p,2}-(DF)(p)|^2.
\]
By~\eqref{eq:Poincare},
$\int_{\B{p}{2\varepsilon}}\!\bigl(F-\bar{F}_{p,2}\bigr)^2 dV \leq 4\,C_P\varepsilon^2\int_{\B{p}{2\varepsilon}}|dF|^2\,dV$, and since
$\bar{F}_{p,2}-(DF)(p)=w(p)^{-1}\int_{\B{p}{\varepsilon}}(\bar{F}_{p,2}-F)\,dV$,
Jensen gives
$|\bar{F}_{p,2}-(DF)(p)|^2\leq w(p)^{-1}\int_{\B{p}{2\varepsilon}}(F-\bar{F}_{p,2})^2\,dV$,
so
\begin{equation*}
    \begin{split}
    \Vol\!\bigl(\B{p}{2\varepsilon}\bigr)\cdot|\bar{F}_{p,2}-(DF)(p)|^2
  & \;\leq\;
  \frac{\Vol(\B{p}{2\varepsilon})}{w(p)}
  \int_{\B{p}{2\varepsilon}}\!\bigl(F-\bar{F}_{p,2}\bigr)^2 dV \\
  & \;\leq\;
  C_d\int_{\B{p}{2\varepsilon}}\!\bigl(F-\bar{F}_{p,2}\bigr)^2 dV \\
  & \;\leq\; 4\,C_d\,C_P\varepsilon^2\int_{\B{p}{2\varepsilon}}|dF|^2\,dV.
    \end{split}
\end{equation*}
Therefore
\[\int_{\B{p}{2\varepsilon}}(F-(DF)(p))^2\,dV
 \leq 8(1+C_d)C_P\varepsilon^2\int_{\B{p}{2\varepsilon}}|dF|^2\,dV.\]
Summing over $p$ and using the multiplicity bound $N'$ for the cover
$\{\B{p}{2\varepsilon}\}$ gives the result.
\end{proof}

\section{Proof of the Main Theorem}
\label{sec:main}

\begin{theorem}
\label{thm:small}
Let $a_X=1/(8C_d^2)$ and $a_M=1/(4B_4\varepsilon^2)$.
For every $\varepsilon$-discretization $X$ of $M$ and every $k<|X|$:
\begin{enumerate}[label=(\roman*)]
\item\label{it:smallX}
if $\lambda_k(X)\leq a_X$, then
$\lambda_k(M)\leq 4NB_2\,\varepsilon^{-2}\,\lambda_k(X)$;
\item\label{it:smallM}
if $\lambda_k(M)\leq a_M$, then
$\lambda_k(X)\leq 4C_dB_3\,\varepsilon^{2}\,\lambda_k(M)$.
\end{enumerate}
\end{theorem}

\begin{proof}
\ref{it:smallX}.
Let $f_0,\ldots,f_k:V\to\R$ be $\ell^2(V,w)$-orthonormal eigenfunctions
of $\Delta_w$ for $\lambda_0(X)\leq\cdots\leq\lambda_k(X)$, and set
$W=\mathrm{span}\{f_0,\ldots,f_k\}$ and
$SW=\mathrm{span}\{Sf_0,\ldots,Sf_k\}$. For nonzero $f\in W$,
Remark~\ref{rem:quadform} gives
$\DD(f,f)\leq\lambda_k(X)\normw{f}^2$. Combining~\eqref{eq:L2D} applied
to $Sf$ with the triangle inequality and
Lemma~\ref{lem:compXMX} yields
\[
  \norml{Sf}
  \;\geq\;
  N^{-1/2}\normw{DSf}
  \;\geq\;
  N^{-1/2}\bigl(\normw{f}-\normw{f-DSf}\bigr)
  \;\geq\;
  N^{-1/2}\normw{f}\bigl(1-(2C_d^2\lambda_k(X))^{1/2}\bigr).
\]
The hypothesis $\lambda_k(X)\leq a_X$ forces
$(2C_d^2\lambda_k(X))^{1/2}\leq\frac{1}{2}$, so
$\norml{Sf}^2\geq\frac{1}{4N}\normw{f}^2>0$. Hence $f\mapsto Sf$
is injective on $W$ and $\dim SW=k+1$. The Rayleigh quotient on $M$
of $F=Sf$ then satisfies, by~\eqref{eq:GradS},
\[
  R(F)
  \;=\;
  \frac{\norml{dF}^2}{\norml{F}^2}
  \;\leq\;
  \frac{B_2\,\varepsilon^{-2}\,\DD(f,f)}{\frac{1}{4N}\normw{f}^2}
  \;\leq\;
  4NB_2\,\varepsilon^{-2}\,\lambda_k(X).
\]
The min-max principle applied to the
$(k+1)$-dimensional subspace $SW\subset W^{1,2}(M)$ gives the claim.

\ref{it:smallM}. The argument is symmetric. Let
$F_0,\ldots,F_k\in L^2(M)$ be orthonormal eigenfunctions of $\Delta_M$
for $\lambda_0(M)\leq\cdots\leq\lambda_k(M)$, and set
$W'=\mathrm{span}\{DF_0,\ldots,DF_k\}\subset\ell^2(V,w)$. For nonzero
$F$ in the span of the $F_i$,
$\norml{dF}^2\leq\lambda_k(M)\norml{F}^2$, and
combining~\eqref{eq:L2S} with the triangle inequality and
Lemma~\ref{lem:compMXM},
\[
  \normw{DF}
  \;\geq\;
  C_d^{-1/2}\norml{SDF}
  \;\geq\;
  C_d^{-1/2}\norml{F}\bigl(1-(B_4\varepsilon^2\lambda_k(M))^{1/2}\bigr)
  \;\geq\;
  \tfrac{1}{2}C_d^{-1/2}\norml{F}>0,
\]
where the last inequality uses $\lambda_k(M)\leq a_M$.
Hence $\dim W'=k+1$, and by~\eqref{eq:GradD} the weighted Rayleigh
quotient of $f=DF$ satisfies
\[
  \Rw(f)
  \;=\;
  \frac{\DD(f,f)}{\normw{f}^2}
  \;\leq\;
  \frac{B_3\,\varepsilon^{2}\,\norml{dF}^2}{\frac{1}{4C_d}\norml{F}^2}
  \;\leq\;
  4C_dB_3\,\varepsilon^{2}\,\lambda_k(M).
\]
The min-max principle on $X$ gives the claim.
\end{proof}

\begin{proof}[Proof of Theorem~\ref{thm:main}]
Fix $k<|X|$ and let $v_0,\ldots,v_k\in V$ be distinct.
\par
\emph{Upper bound.} If $\lambda_k(X)\leq a_X$,
Theorem~\ref{thm:small}\ref{it:smallX} gives
$\lambda_k(M)\leq 4NB_2\,\varepsilon^{-2}\lambda_k(X)$. If instead
$\lambda_k(X)>a_X$, let $\psi_i$ be the first Dirichlet eigenfunction
of $\B{v_i}{\varepsilon/2}$, extended by zero. The balls
$\B{v_i}{\varepsilon/2}$ are pairwise disjoint, so the $\psi_i$ are
$L^2$-orthogonal, and Cheng's theorem~\cite{Cheng1975} gives
$R(\psi_i)\leq\lambda_1^{-\kappa}(\varepsilon/2)=A\varepsilon^{-2}$,
where $\lambda_1^{-\kappa}(\varepsilon/2)$ is the first Dirichlet
eigenvalue of a ball of radius $\varepsilon/2$ in the simply connected
$n$-dimensional space form of curvature $-\kappa$, and
$A=\varepsilon^2\lambda_1^{-\kappa}(\varepsilon/2)$. The
min-max principle over $\operatorname{span}\{\psi_0,\ldots,\psi_k\}$
yields $\lambda_k(M)\leq A\varepsilon^{-2}<(A/a_X)\,\varepsilon^{-2}
\lambda_k(X)=8C_d^2A\,\varepsilon^{-2}\lambda_k(X)$.
\par
\emph{Lower bound.} If $\lambda_k(M)\leq a_M$,
Theorem~\ref{thm:small}\ref{it:smallM} gives
$\lambda_k(X)\leq 4C_dB_3\,\varepsilon^{2}\lambda_k(M)$. If instead
$\lambda_k(M)>a_M$, Lemma~\ref{lem:BG}\ref{it:Rw} gives
$\lambda_k(X)\leq 2N<(2N/a_M)\lambda_k(M)=8NB_4\,\varepsilon^2\lambda_k(M)$.
\par
Combining the two cases on each side, Theorem~\ref{thm:main} holds with
\begin{equation}
\label{eq:main-constants}
  c_1=\frac{\varepsilon^{-2}}{\max\{4C_dB_3,\;8NB_4\}},
  \qquad
  c_2=\varepsilon^{-2}\max\{4NB_2,\;8C_d^2A\}.
\end{equation}
\end{proof}

A natural consequence is that the weighted spectrum is an invariant of
the manifold at scale $\varepsilon$, independent of the choice of
discretization.

\begin{corollary}
\label{cor:independence}
Let $X$ and $Y$ be two $\varepsilon$-discretizations of $M$, and let
$c_1,c_2$ be the constants of Theorem~\ref{thm:main}. Then
\[
  \frac{c_1}{c_2}\,\lambda_k(Y)\;\leq\;\lambda_k(X)\;\leq\;
  \frac{c_2}{c_1}\,\lambda_k(Y)
  \qquad\text{for all }k<\min(|X|,|Y|).
\]
The ratio $c_2/c_1$ depends only on $n$, $\kappa$ and $\varepsilon$.
\end{corollary}

\begin{proof}
Theorem~\ref{thm:main} applied to $(M,X)$ and $(M,Y)$ gives
$\lambda_k(X)\leq c_1^{-1}\lambda_k(M)\leq c_1^{-1}c_2\,\lambda_k(Y)$ and
$\lambda_k(X)\geq c_2^{-1}\lambda_k(M)\geq c_2^{-1}c_1\,\lambda_k(Y)$.
\end{proof}

\subsection{Towers of coverings}
\label{subsec:tower}

The spectral comparison of Theorem~\ref{thm:main} is stable under
finite-sheeted coverings, with constants uniform along the tower. Let
$(M^n,g)$ satisfy $\Ric(M,g)\geq-(n-1)\kappa g$, fix
$0<\varepsilon<\tfrac{1}{2}\Inj(M)$, and let $X=(V,E)$ be an
$\varepsilon$-discretization of $M$. For each $i\geq 1$, let
$\pi_i:M_i\to M$ be a finite-sheeted Riemannian covering with
pulled-back metric $g_i=\pi_i^*g$, and let $X_i=(V_i,E_i)$ be the
associated $\varepsilon$-discretization of $M_i$ with vertex set $V_i=\pi_i^{-1}(V)$ and edge
set $\{p,q\}\in E_i\iff 0<d_{M_i}(p,q)<3\varepsilon$, carrying the
weights $w_i(p)=\Vol(\B{p}{\varepsilon})$ and
$\mu_i(p,q)=\min(w_i(p),w_i(q))$.

Each $(M_i,g_i)$ satisfies $\Ric\geq-(n-1)\kappa g_i$ because $\pi_i$
is a local isometry. Distinct vertices $p,q\in V_i$ with
$\pi_i(p)\neq\pi_i(q)$ have $d_{M_i}(p,q)\geq d_M(\pi_i(p),\pi_i(q))
\geq\varepsilon$, while distinct vertices with $\pi_i(p)=\pi_i(q)$ satisfy
$d_{M_i}(p,q)\geq 2\Inj(M)>\varepsilon$, since a minimizing geodesic
from $p$ to $q$ projects to a geodesic loop at $\pi_i(p)$, which has
length at least $2\Inj(\pi_i(p))$; maximality
follows by lifting a minimizing geodesic to any prescribed endpoint.
The $\varepsilon$-balls lift isometrically, so $w_i(p)=w(\pi_i(p))$.
Applying Theorem~\ref{thm:main} to each pair $(M_i,X_i)$ and using the
fact that its constants depend only on $n,\kappa,\varepsilon$, we obtain

\begin{theorem}
\label{thm:tower}
Let $c_1,c_2$ be the constants of Theorem~\ref{thm:main}. For all
$i\geq 1$ and all $k<|X_i|$,
\[
  c_1\,\lambda_k(X_i)\;\leq\;\lambda_k(M_i)\;\leq\;c_2\,\lambda_k(X_i).
\]
In particular, $\lambda_k(M_i)\to 0$ if and only if
$\lambda_k(X_i)\to 0$ as $i\to\infty$.
\end{theorem}

\begin{remark}
\label{rem:brooks}
The volumes $\Vol(M_i)$ may grow unboundedly while the local geometry
stays controlled. Let $\Gamma_i$ be the Schreier graph of
$\pi_1(M_i)\subset\pi_1(M)$ with respect to a fixed finite generating
set of $\pi_1(M)$. Since $\varepsilon<\frac12\Inj(M)$, the weights of
$X_i$ are uniformly comparable to constants, and $X_i$ is roughly
isometric to $\Gamma_i$ with constants independent of~$i$. Combining
Theorem~\ref{thm:tower} with the rough-isometry invariance of
$\lambda_1$ for graphs of bounded degree (cf.~\cite{Mantuano2005}) and
the discrete Cheeger inequality gives
$\inf_i\lambda_1(M_i)>0$ if and only if $\inf_i h(\Gamma_i)>0$, which
is Brooks' theorem~\cite{Brooks1986}.
\end{remark}

\section{Stability under Measured Gromov--Hausdorff Convergence}
\label{sec:applications}
In this section we prove a spectral stability theorem under measured
Gromov--Hausdorff convergence. Mantuano~\cite[Theorem~5.1]{Mantuano2005}
established such a result under a uniform injectivity radius bound;
we replace that bound with the strictly weaker non-collapsing
hypothesis $\Vol(\B{p}{\varepsilon})\geq v_0$.

\begin{definition}
\label{def:rough-isom}
A map $\Phi:(X_1,d_1)\to(X_2,d_2)$ is a \emph{rough isometry} with
constants $(a,b,\tau)$, $a\geq 1$, $b,\tau\geq 0$, if
\[
  a^{-1}d_1(x,y)-b \;\leq\; d_2(\Phi(x),\Phi(y))
  \;\leq\; a\,d_1(x,y)+b
  \qquad\text{for all }x,y\in X_1,
\]
and $\bigcup_{x\in X_1}B(\Phi(x),\tau)=X_2$.
\end{definition}

We use
the following Prokhorov-type formulation of measured
Gromov--Hausdorff closeness; compare Fukaya~\cite{Fukaya1987} and
Sturm~\cite{Sturm2006}. For $A\subset Z$ and $\eta>0$ we write
$U_\eta(A)=\{z\in Z:d_Z(z,A)<\eta\}$.

\begin{definition}
\label{def:mGH}
Let $(M,d_M,\mu_M)$ and $(N,d_N,\mu_N)$ be compact metric measure
spaces. We write $d_{\mathrm{mGH}}(M,N)<\eta$ if there exist a metric
space $(Z,d_Z)$ and isometric embeddings $f:M\hookrightarrow Z$,
$g:N\hookrightarrow Z$ such that
\begin{enumerate}[label=(\roman*)]
\item $d_H^Z(f(M),g(N))<\eta$, where $d_H^Z$ is the Hausdorff
distance in $Z$;
\item\label{it:mGH-measure} for every Borel $A\subset Z$,
\[
  f_*\mu_M(A)\leq g_*\mu_N\bigl(U_\eta(A)\bigr)+\eta
  \quad\text{and}\quad
  g_*\mu_N(A)\leq f_*\mu_M\bigl(U_\eta(A)\bigr)+\eta.
\]
\end{enumerate}
\end{definition}

\begin{lemma}
\label{lem:weight-comp}
Let $(M,g_M)$ and $(N,g_N)$ be compact Riemannian $n$-manifolds with
$\Ric\geq-(n-1)\kappa g$, equipped with their volume measures. Fix
$v_0>0$, $\varepsilon>0$, $\eta\in(0,\varepsilon/2)$, and assume
$\Vol(\B{p}{\varepsilon})\geq v_0$ for every $p\in M\cup N$. Suppose
$d_{\mathrm{mGH}}(M,N)<\eta$ via embeddings $f:M\hookrightarrow Z$,
$g:N\hookrightarrow Z$. Then for any $p\in M$ and $q\in N$ with
$d_Z(f(p),g(q))<\delta$ for some $\delta>0$,
\[
  K^{-1}\,w_M(p)\;\leq\; w_N(q)\;\leq\; K\,w_M(p),
\]
with $K=K(n,\kappa,\varepsilon,\eta,\delta,v_0)$.
\end{lemma}

\begin{proof}
Apply condition~\ref{it:mGH-measure} of Definition~\ref{def:mGH} to
the Borel set $A=f\bigl(B_M(p,\varepsilon-\eta)\bigr)$. Since $f$ is an
isometric embedding, $f_*\mu_M(A)=\mu_M(B_M(p,\varepsilon-\eta))$ and
$U_\eta(A)\subset B_Z(f(p),\varepsilon)$, so
\[
  \mu_M\bigl(B_M(p,\varepsilon-\eta)\bigr)
  \;\leq\;
  g_*\mu_N\bigl(B_Z(f(p),\varepsilon)\bigr)+\eta.
\]
Any $q'\in N$ with $g(q')\in B_Z(f(p),\varepsilon)$ satisfies
$d_Z(g(q'),g(q))<\varepsilon+\delta$, hence
$d_N(q',q)<\varepsilon+\delta$ by isometry of~$g$. Set
$R=\varepsilon+\delta$; then
\begin{equation}
\label{eq:weight-containment}
  \mu_M\bigl(B_M(p,\varepsilon-\eta)\bigr)
  \;\leq\;
  \mu_N\bigl(B_N(q,R)\bigr)+\eta.
\end{equation}
Bishop--Gromov at~$p$ with radii $\varepsilon-\eta<\varepsilon$
gives the lower bound
\[
  \mu_M\bigl(B_M(p,\varepsilon-\eta)\bigr)
  \;\geq\;
  w_M(p)\cdot\frac{\Vk(\varepsilon-\eta)}{\Vk(\varepsilon)}.
\]
For the right-hand side, Bishop--Gromov at~$q$ with radii
$\varepsilon\leq R$ gives
$\mu_N(B_N(q,R))\leq w_N(q)\cdot\Vk(R)/\Vk(\varepsilon)$.
Substituting into~\eqref{eq:weight-containment} and using
$w_N(q)\geq v_0$ to absorb the additive~$\eta$:
\[
  w_M(p)\cdot\frac{\Vk(\varepsilon-\eta)}{\Vk(\varepsilon)}
  \;\leq\;
  w_N(q)\cdot\frac{\Vk(R)}{\Vk(\varepsilon)}\,+\,\eta
  \;\leq\;
  w_N(q)\cdot\frac{\Vk(R)}{\Vk(\varepsilon)}
  \bigl(1+\eta/v_0\bigr).
\]
This gives $w_N(q)\geq K^{-1}w_M(p)$ with
\[
  K^{-1}\;=\;
  \frac{\Vk(\varepsilon-\eta)}{(1+\eta/v_0)\,\Vk(R)},
  \qquad R=\varepsilon+\delta.
\]
Swapping $M$ and $N$ (with the same proximity~$\delta$) gives the
reverse inequality.
\end{proof}

We now obtain a spectral comparison between roughly isometric
$\varepsilon$-discretizations.
\begin{proposition}
\label{prop:weighted-rough}
Let $M$ and $N$ be compact Riemannian $n$-manifolds with
$\Ric\geq-(n-1)\kappa g$, and let $X_M$, $X_N$ be their
$\varepsilon$-discretizations. Let $\Phi:V(X_M)\to V(X_N)$ be a
rough isometry with constants $(a,b,\tau)$ satisfying
\begin{equation}
\label{eq:weight-comp}
  K^{-1}\,w_M(p)\;\leq\; w_N(\Phi(p))\;\leq\; K\,w_M(p)
  \qquad\text{for all }p\in V(X_M).
\end{equation}
Then there exist $\tilde c,\tilde C>0$ depending only on
$n,\kappa,\varepsilon,a,b,\tau,K$ such that
\[
  \tilde c\,\lambda_k(X_N)\;\leq\;\lambda_k(X_M)\;\leq\;
  \tilde C\,\lambda_k(X_N)
  \qquad\text{for all }k<\min(|X_M|,|X_N|).
\]
\end{proposition}

\begin{proof}
By symmetry it suffices to prove the upper bound. Define $\Phi^*:\ell^2(V(X_N),w_N)\to\ell^2(V(X_M),w_M)$ by
$(\Phi^*f)(p)=f(\Phi(p))$. For
$p,p'\in\Phi^{-1}(q)$, we have $d_M(p,p')\leq ab$ since
$d_N(\Phi(p),\Phi(p'))=0$, and Lemma~\ref{lem:BG}\ref{it:count} bounds the
number of $\varepsilon$-separated points in any ball of radius~$ab$
by a constant $m_\Phi=m_\Phi(a,b,n,\kappa,\varepsilon)$. Therefore,
using~\eqref{eq:weight-comp},
\[\|\Phi^*f\|_{w_M}^2
  \;=\;\sum_{q\in\mathrm{Im}(\Phi)}f(q)^2\!\sum_{p\in\Phi^{-1}(q)}\!w_M(p)
  \;\leq\;
  K\,m_\Phi\,\|f\|_{w_N}^2.\]
For the reverse inequality, let $q\in\mathrm{Im}(\Phi)$ and pick
any $p\in\Phi^{-1}(q)$. By~\eqref{eq:weight-comp},
$w_M(p)\geq K^{-1}w_N(q)$; summing over $q\in\mathrm{Im}(\Phi)$
gives
\begin{equation}
\label{eq:pullback-norm-covered}
  \|\Phi^*f\|_{w_M}^2
  \;\geq\;
  K^{-1}\!\sum_{q\in\mathrm{Im}(\Phi)}f(q)^2\,w_N(q).
\end{equation}

Let $U=V(X_N)\setminus\mathrm{Im}(\Phi)$ and define
\[
  \mathcal{U}(f)\;=\;\sum_{q\in U}f(q)^2\,w_N(q).
\]
By $\tau$-density of $\Phi$, for each $q \in U$ there exists
$p_q \in V(X_M)$ with $d_N(\Phi(p_q), q) \le \tau$; set
$\sigma(q) = \Phi(p_q) \in \mathrm{Im}(\Phi)$. Let
$\gamma : [0, \ell] \to N$ be a minimizing geodesic from $q$ to
$\sigma(q)$, parametrized by arc length, with
$\ell = d_N(q, \sigma(q)) \le \tau$. Set $s = \varepsilon/2$, let
$J = \lceil \ell/s \rceil$, and consider the sample points
$\gamma(0), \gamma(s), \gamma(2s), \ldots, \gamma(Js) = \sigma(q)$
(with the convention $\gamma(t) = \sigma(q)$ for $t \ge \ell$).
Since $V(X_N)$ is a maximal $\varepsilon$-separated subset of $N$,
for each $j = 1, \ldots, J-1$ there exists $q_j \in V(X_N)$ with
$d_N(q_j, \gamma(js)) < \varepsilon$; set $q_0 = q$ and
$q_{L_0} = \sigma(q)$, where $L_0 = J \le \lceil 2\tau/\varepsilon \rceil + 1$.
By the triangle inequality,
\[
  d_N(q_j, q_{j+1})
  \le d_N(q_j, \gamma(js)) + s + d_N(\gamma((j+1)s), q_{j+1})
  < 2\varepsilon + s
  = \tfrac{5\varepsilon}{2}
  < 3\varepsilon,
\]
so $\{q_j, q_{j+1}\} \in E(X_N)$ for each $j$. The sequence
\[
  q = q_0, q_1, \ldots, q_{L_0} = \sigma(q)
\]
is therefore a path in $X_N$ of combinatorial length $L_0$.
Deleting cycles if necessary, we may assume it is a simple path. By the Cauchy--Schwarz inequality,
\[
  |f(q)-f(\sigma(q))|^2
  \;\leq\;
  L_0\sum_{j=0}^{L_0-1}|f(q_j)-f(q_{j+1})|^2,
\]
and the inequality $a^2\leq 2b^2+2(a-b)^2$ yields
\begin{equation}
\label{eq:pointwise-mm}
  f(q)^2
  \;\leq\;
  2\,f(\sigma(q))^2
  +
  2L_0\sum_{j=0}^{L_0-1}|f(q_j)-f(q_{j+1})|^2.
\end{equation}
Iterating Lemma~\ref{lem:BG}\ref{it:comp} along the path gives
$w_N(q)\leq C_d^{L_0}w_N(\sigma(q))$ and
$w_N(q)\leq C_d^{L_0+1}\mu_N(q_j,q_{j+1})$ for each $j$.
Multiplying~\eqref{eq:pointwise-mm} by $w_N(q)$ and applying these
bounds,
\begin{equation}
\label{eq:weighted-pointwise}
  f(q)^2\,w_N(q)
  \;\leq\;
  2C_d^{L_0}\,f(\sigma(q))^2\,w_N(\sigma(q))
  +
  2L_0 C_d^{L_0+1}\sum_{j=0}^{L_0-1}\mu_N(q_j,q_{j+1})|f(q_j)-f(q_{j+1})|^2.
\end{equation}
We now sum~\eqref{eq:weighted-pointwise} over $q\in U$. Applying Lemma~\ref{lem:BG}\ref{it:count} gives
$|\{q \in U : \sigma(q) = q'\}| \le m_1$ for each $q' \in \mathrm{Im}(\Phi)$,
and $|\{q \in U : \text{path of } q \text{ uses } \{r, r'\}\}| \le m_2$
for each $\{r, r'\} \in E(X_N)$, with $m_1, m_2$ depending only on
$n, \kappa, \varepsilon, \tau$. Therefore, by~\eqref{eq:pullback-norm-covered},
\begin{equation}
\label{eq:mm-clean}
  \begin{split}
  \mathcal{U}(f)
  &\;\leq\;
  2C_d^{L_0}m_1\!\sum_{q'\in\mathrm{Im}(\Phi)}\!f(q')^2 w_N(q')
  + 2L_0 C_d^{L_0+1}m_2\,\DD_N(f,f)\\
  &\;\leq\; K_1 K\,\|\Phi^*f\|_{w_M}^2 + K_2\,\DD_N(f,f),
  \end{split}
\end{equation}
where $K_1=2C_d^{L_0}m_1$ and $K_2=2L_0 C_d^{L_0+1}m_2$, both
depending only on $n,\kappa,\varepsilon,\tau$. Writing $\|f\|_{w_N}^2
=\sum_{q'\in\mathrm{Im}(\Phi)}f(q')^2 w_N(q')+\mathcal{U}(f)$ and applying
~\eqref{eq:pullback-norm-covered} and~\eqref{eq:mm-clean},
\[
  \|f\|_{w_N}^2
  \;\leq\;
  K(1+K_1)\,\|\Phi^*f\|_{w_M}^2 + K_2\,\DD_N(f,f).
\]
Rearranging, we have
\begin{equation}
\label{eq:pullback-norm-lower}
  \|\Phi^*f\|_{w_M}^2
  \;\geq\;
  K_3^{-1}\|f\|_{w_N}^2 - K_3^{-1}K_2\,\DD_N(f,f)
\end{equation}
where $K_3=K(1+K_1)$.
We now turn to the Dirichlet form. Fix $\{p,p'\}\in E(X_M)$, so
$d_M(p,p')<3\varepsilon$ and hence
$d_N(\Phi(p),\Phi(p'))\leq 3a\varepsilon+b$. Using the same construction as for $\gamma$, $\Phi(p)$ and $\Phi(p')$ are
joined in $X_N$ by a path
$\Phi(p)=r_0,r_1,\ldots,r_{L_1}=\Phi(p')$ of length
$L_1=\lceil(3a\varepsilon+b)/\varepsilon\rceil+1$, and Cauchy--Schwarz gives
\[
  |f(\Phi(p))-f(\Phi(p'))|^2
  \;\leq\;
  L_1\sum_{j=0}^{L_1-1}|f(r_j)-f(r_{j+1})|^2.
\]
The edge weight satisfies
$\mu_M(p,p')\leq w_M(p)\leq K\,w_N(\Phi(p))$
by~\eqref{eq:weight-comp}, and iterating Lemma~\ref{lem:BG}\ref{it:comp}
along the path gives
$w_N(\Phi(p))\leq C_d^{L_1+1}\,\mu_N(r_j,r_{j+1})$ for each $j$.
Each edge of $X_N$ lies on the path of at most $m_\Phi\cdot N$
edges of $X_M$, where $N$ is the degree bound of
Lemma~\ref{lem:BG}\ref{it:count}; absorbing these multiplicities yields
\begin{equation}
\label{eq:pullback-dirichlet}
  \DD_M(\Phi^*f,\Phi^*f)\;\leq\; K_4\,\DD_N(f,f),
\end{equation}
with $K_4=K_4(n,\kappa,\varepsilon,a,b,\tau,K)$.

The min-max argument now follows the template of
Theorem~\ref{thm:small}. Set $a_0 = (4K_2)^{-1}$. If
$\lambda_k(X_N) \le a_0$, taking $W$ to be the span of the first
$k+1$ eigenfunctions of $\Delta_{w,N}$ and combining
\eqref{eq:pullback-norm-lower} with $\DD_N(f,f) \le \lambda_k(X_N)\|f\|_{w_N}^2$
on $W$ gives $\|\Phi^* f\|_{w_M}^2 \ge \tfrac{3}{4} K_3^{-1} \|f\|_{w_N}^2$
for $f \in W \setminus \{0\}$, so $\Phi^*W$ is $(k+1)$-dimensional and
\eqref{eq:pullback-dirichlet} yields
$\Rw(\Phi^* f) \le \tfrac{4}{3} K_3 K_4 \, \lambda_k(X_N)$. The min-max
principle gives $\lambda_k(X_M) \le \tfrac{4}{3} K_3 K_4 \, \lambda_k(X_N)$.
The case $\lambda_k(X_N) > a_0$ follows from $\lambda_k(X_M) \le 2N$
(Lemma~\ref{lem:BG}\ref{it:Rw}). Taking
$\tilde C = \max(\tfrac{4}{3} K_3 K_4, \, 2N/a_0)$ completes the proof.
\end{proof}

The proof of Theorem~\ref{thm:mGH} now reduces to verifying that
the natural map produced by the mGH embedding is a rough isometry
in the sense of Definition~\ref{def:rough-isom}, including its
$\tau$-density.
\begin{proof}[Proof of Theorem~\ref{thm:mGH}]
Fix $\varepsilon$-discretizations $X_M, X_N$ of $M$ and $N$, and let
$f : M \hookrightarrow Z$, $g : N \hookrightarrow Z$ be embeddings
with $d_H^Z(f(M), g(N)) < \eta$. For
each $p \in V(X_M)$, there exists
$p' \in N$ with $d_Z(f(p), g(p')) < \eta$, and $\varepsilon$-density
of $V(X_N)$ in $N$ gives $\Phi(p) \in V(X_N)$ with
$d_N(\Phi(p), p') < \varepsilon$; hence
\begin{equation}
\label{eq:prox}
  d_Z(f(p), g(\Phi(p))) < \eta + \varepsilon.
\end{equation}
Applying \eqref{eq:prox} to $p_1, p_2 \in V(X_M)$ and using that
$f, g$ are isometric embeddings,
$|d_N(\Phi(p_1), \Phi(p_2)) - d_M(p_1, p_2)| \le 2(\eta + \varepsilon)$.
For the density condition, given $q \in V(X_N)$, there exists $p_0 \in M$ with $d_Z(f(p_0), g(q)) < \eta$,
$\varepsilon$-density of $V(X_M)$ in $M$ gives $p \in V(X_M)$
with $d_M(p, p_0) < \varepsilon$, and combining with \eqref{eq:prox}
yields $d_N(\Phi(p), q) < 2(\eta + \varepsilon)$. Thus $\Phi$ is a
rough isometry with
constants $(1, 2(\eta + \varepsilon), 2(\eta + \varepsilon))$.

Lemma~\ref{lem:weight-comp} applied at each pair $(p, \Phi(p))$
via \eqref{eq:prox} with $\delta = \eta + \varepsilon$
gives $K^{-1} w_M(p) \le w_N(\Phi(p)) \le K w_M(p)$, with
$K = K(n, \kappa, \varepsilon, \eta, v_0)$ since the radius
is $R=\varepsilon+\delta=2\varepsilon+\eta$. Hence Proposition~\ref{prop:weighted-rough} applies with
$(a,b,\tau)=(1,2(\eta+\varepsilon),2(\eta+\varepsilon))$ and gives
$\tilde c\,\lambda_k(X_N)\leq\lambda_k(X_M)\leq\tilde C\,\lambda_k(X_N)$.
Combining this with Theorem~\ref{thm:main} for $(M,X_M)$ and $(N,X_N)$,
\[
  \lambda_k(M)\leq c_2\lambda_k(X_M)\leq c_2\tilde C\,\lambda_k(X_N)
  \leq \frac{c_2\tilde C}{c_1}\,\lambda_k(N),
  \qquad
  \lambda_k(N)\leq c_2\lambda_k(X_N)\leq \frac{c_2}{\tilde c}\,\lambda_k(X_M)
  \leq \frac{c_2}{c_1\tilde c}\,\lambda_k(M),
\]
so Theorem~\ref{thm:mGH} holds with $c=c_1/(c_2\tilde C)$ and
$C=c_2/(c_1\tilde c)$, which depend only on
$n,\kappa,\varepsilon,\eta,v_0$.
\end{proof}

\begin{corollary}
\label{cor:relative-noncollapsing}
Let $(M,g_M)$ and $(N,g_N)$ be compact Riemannian $n$-manifolds with
$\Ric\geq-(n-1)\kappa g$, equipped with their volume measures.
Fix $\varepsilon>0$ and $\eta\in(0,\varepsilon/2)$, and suppose
$d_{\mathrm{mGH}}(M,N)<\eta$ via embeddings $f:M\hookrightarrow Z$,
$g:N\hookrightarrow Z$. Assume there exists $K>0$ such that for every
$p\in M$ and $q\in N$ with $d_Z(f(p),g(q))<\eta$,
\begin{equation}
\label{eq:relative-weight}
  K^{-1}\,w_M(p)\;\leq\;w_N(q)\;\leq\;K\,w_M(p).
\end{equation}
Then there exist constants $c,C>0$ depending only on
$n,\kappa,\varepsilon,\eta,K$ such that
\[
  c\,\lambda_k(M)\;\leq\;\lambda_k(N)\;\leq\;C\,\lambda_k(M)
  \qquad\text{for all }k<\min(|X_M|,|X_N|),
\]
where $X_M,X_N$ are any $\varepsilon$-discretizations of $M$ and $N$.
In particular, no absolute lower bound on ball volumes is required.
The result applies whenever $M$ and $N$ collapse at comparable rates.
\end{corollary}
\begin{proof}
The proof follows that of Theorem~\ref{thm:mGH} with
\eqref{eq:relative-weight} replacing Lemma~\ref{lem:weight-comp},
the $\varepsilon$-gap absorbed via Lemma~\ref{lem:BG}\ref{it:comp}.
\end{proof}
\section{An Application: The Schoen--Wolpert--Yau Inequality}
\label{sec:examples}

In this section we prove Proposition~\ref{prop:intro-swy}. Fix $r_0>0$,
set $\varepsilon_0=\min\{1/5,r_0\}$, and fix
$\varepsilon\in(0,\varepsilon_0)$. Let $\Sigma_\ell$, $\gamma_\ell$ be as
in Proposition~\ref{prop:intro-swy}, let $X_\ell=(V,E)$ be an
$\varepsilon$-discretization of $\Sigma_\ell$, and write
$\DD=\DD_{X_\ell}$. We take $\ell_1=\varepsilon/\cosh2$ in
Proposition~\ref{prop:intro-swy} and assume $\ell\leq\ell_1$.

\subsection{The collar and the weights}
By the Collar Lemma~\cite[\S4.1]{Buser1992}, $\gamma_\ell$ has an
embedded collar $\mathcal C(\gamma_\ell)=\{|\rho|<\omega\}$,
$\sinh\omega=1/\sinh(\ell/2)$, with metric
$d\rho^2+\ell^2\cosh^2\!\rho\,dt^2$, $t\in\R/\Z$, where $\rho$ is the
signed distance to $\gamma_\ell$. Set
\[
  \rho_c=\operatorname{arccosh}(\varepsilon/\ell),\qquad
  T=\{x\in\mathcal C(\gamma_\ell):|\rho(x)|<\rho_c\}.
\]
Since $\ell\leq\ell_1$, we have $\rho_c\geq2$; since
$e^{\rho_c}\leq2\varepsilon/\ell$ and $e^{\omega}\geq2/\ell$, we have
$\omega-\rho_c\geq\log(1/\varepsilon)>3\varepsilon$. Each fibre
$\{\rho=s\}$ has length $\ell\cosh s$, which is less than $\varepsilon$
in~$T$.

\emph{Vertices in $T$.} Each fibre in $T$ has diameter less than
$\varepsilon/2$, so $d(x,y)\leq|\rho(x)-\rho(y)|+\varepsilon/2$ for
$x,y\in T$, and distinct vertices in $T$ satisfy
$|\rho(p)-\rho(q)|\geq\varepsilon/2$. Every point of $\Sigma_\ell$ lies
within $\varepsilon$ of $V$, and a vertex $v$ within $\varepsilon$ of a point
$x\in T$ with $|\rho(x)|\leq\rho_c-\varepsilon$ lies in $T$ and
satisfies $|\rho(v)-\rho(x)|<\varepsilon$. Applied to the points with
$\rho=\pm\varepsilon$ and $\rho=\pm(\rho_c-\varepsilon)$, this gives
vertices of $T$ with $\pm\rho\in(0,2\varepsilon)$ and with
$\pm\rho\in(\rho_c-2\varepsilon,\rho_c)$. If $p,q\in V\cap T$ are
consecutive in the ordering by $\rho$ and
$|\rho(p)-\rho(q)|>2\varepsilon$, the point $x$ with
$\rho(x)=\frac12(\rho(p)+\rho(q))$ on a radial geodesic satisfies
$|\rho(x)|<\rho_c-\varepsilon$, so a vertex within $\varepsilon$ of $x$
lies strictly between $p$ and $q$, a contradiction. Hence,
ordering $V\cap T$ by $\rho$, consecutive vertices satisfy
$\varepsilon/2\leq|\rho(p)-\rho(q)|\leq2\varepsilon$ and are adjacent
in~$X_\ell$. Let
\[
  T^+=\{v\in V:0\leq\rho(v)<\rho_c\},\qquad
  T^-=\{v\in V:-\rho_c<\rho(v)<0\},
\]
let $c_\pm$ and $q_\pm$ be the vertices of $T^\pm$ with the smallest and
the largest $|\rho|$, respectively. Then $c_+\sim c_-$, as consecutive vertices, and
$\pm\rho(q_\pm)\geq\rho_c-2\varepsilon$. A vertex $p_\pm$ within
$\varepsilon$ of the point at $\rho=\rho(q_\pm)\pm\varepsilon$ on the
geodesic through $q_\pm$ orthogonal to the fibres satisfies
$p_\pm\in V\setminus T$ and $d(p_\pm,q_\pm)<2\varepsilon$, so
$p_\pm\sim q_\pm$.

\emph{Weights in $T$.} For $v\in T$, $\B{v}{\varepsilon}$ lies in
$\{|\rho-\rho(v)|<\varepsilon\}$, whose area is
$2\ell\cosh\rho(v)\sinh\varepsilon$. Every fibre in
$\{|\rho-\rho(v)|<\varepsilon/4\}$ has length at most
$\varepsilon e^{\varepsilon/4}$, so each point of this set is joined to
$v$ by a path of length at most
$\varepsilon/4+\tfrac12\varepsilon e^{\varepsilon/4}<\varepsilon$. Hence
\begin{equation}
\label{eq:collar-weight}
  \alpha_1\,\ell\cosh\rho(v)\;\le\;w(v)\;\le\;\alpha_2\,\ell\cosh\rho(v)
  \quad(v\in T),
  \qquad
  \alpha_1=2\sinh\tfrac\varepsilon4,\quad \alpha_2=2\sinh\varepsilon.
\end{equation}

\emph{Weights outside $T$.} Let $v\in V\setminus T$. A geodesic loop at
$v$ not contained in $\mathcal C(\gamma_\ell)$ is a non-contractible loop
through a point $y$ with $\Inj(y)\geq r_0$, by~\eqref{eq:intro-thick},
so it has length at least $2r_0>\varepsilon/2$, since a
non-contractible closed curve through $y$ has length at least
$2\Inj(y)$~\cite[Ch.~4]{Buser1992}. A geodesic loop at $v$
inside $\mathcal C(\gamma_\ell)$ is freely homotopic to a nonzero power
of $\gamma_\ell$, say $\gamma_\ell^k$, so its length $L$ satisfies
$\sinh(L/2)=\sinh(k\ell/2)\cosh\rho(v)$~\cite[Ch.~2]{Buser1992}; hence
$\sinh(L/2)\geq\sinh(\ell/2)\cosh\rho(v)\geq\varepsilon/2$, and
$L\geq\varepsilon/2$. Thus $\Inj(v)\geq\varepsilon/4$, and since an
embedded ball of radius $s$ in curvature $-1$ has area
$2\pi(\cosh s-1)$, Bishop--Gromov gives
\begin{equation}
\label{eq:thick-weight}
  \tfrac{\pi}{16}\,\varepsilon^2\;\le\;w(v)\;\le\;2\pi\,\varepsilon^2
  \qquad (v\in V\setminus T).
\end{equation}

\emph{Sides.} If $\gamma_\ell$ separates, let $\Sigma^\pm$ be the
closures of the components of $\Sigma_\ell\setminus\gamma_\ell$ and
$S^\pm=(V\setminus T)\cap\Sigma^\pm$. If $\gamma_\ell$ is
non-separating, put $\Sigma^\pm=\Sigma_\ell$ and $S^\pm=V\setminus T$.
Since $\{\B{v}{\varepsilon}\}$ has multiplicity at most $N$
(Lemma~\ref{lem:BG}\ref{it:count}) and
$\operatorname{Area}(\Sigma_\ell)=4\pi$,
\begin{equation}
\label{eq:vol-upper}
  \vol_w(U)\;\le\;4\pi N\qquad(U\subset V).
\end{equation}
In the separating case, $\operatorname{Area}(\Sigma^\pm)=2\pi$ by
Gauss--Bonnet; every point of $\Sigma^\pm$ outside
$\{|\rho|<\rho_c+\varepsilon\}$ lies within $\varepsilon$ of a vertex of
$S^\pm$, and $\{x\in\Sigma^\pm:|\rho(x)|<\rho_c+\varepsilon\}$ has area
$\ell\sinh(\rho_c+\varepsilon)\leq\varepsilon e^\varepsilon<1$, so
\begin{equation}
\label{eq:vol-lower}
  \vol_w(S^\pm)\;\geq\;\pi .
\end{equation}

\subsection{Two inequalities}

\begin{lemma}
\label{lem:side-poincare}
Let $S=S^+$ or $S=S^-$, and $m=64\,\varepsilon^{-2}$. Then for every
$u:V\to\R$,
\[
  \sum_{S}w\,(u-\bar u_S)^2\;\le\;\alpha_3\,\DD(u,u),
  \qquad \alpha_3=512\,m^3,
\]
where $\bar u_S=\vol_w(S)^{-1}\sum_{S}uw$.
\end{lemma}

\begin{proof}
The balls $\B{v}{\varepsilon/4}$, $v\in S$, are disjoint and have area
at least $\pi\varepsilon^2/16$, so $|S|\leq m$.

The subgraph of $X_\ell$ induced on $S$ is connected. Indeed, the set
$\Omega=\Sigma^\pm\setminus\{|\rho|<\rho_c+2\varepsilon\}$ is connected,
since $\rho_c+2\varepsilon<\omega$, and every vertex within $\varepsilon$
of a point of $\Omega$ lies in $S$. Given $v,v'\in S$, choose
$x,x'\in\Omega$ with $d(v,x),d(v',x')\leq2\varepsilon$, points
$x=x_0,\ldots,x_J=x'$ on a path in $\Omega$ with
$d(x_i,x_{i+1})\leq\varepsilon$, and vertices $v_i$ with
$d(v_i,x_i)<\varepsilon$. Then $v,v_0,\ldots,v_J,v'$ is a path in $S$,
since consecutive distances are less than $3\varepsilon$.

For $p,q\in S$ fix a simple path $\gamma_{pq}$ in $S$ from $p$ to $q$.
Each has at most $m$ edges, each edge lies on at most $m^2$ such paths,
and each edge $e$ of $S$ has $\mu_e\geq\pi\varepsilon^2/16$
by~\eqref{eq:thick-weight}. By Cauchy--Schwarz,
\[
  \sum_{p,q\in S}\bigl(u(p)-u(q)\bigr)^2
  \;\le\; m\sum_{p,q\in S}\sum_{e\in\gamma_{pq}}(\delta_eu)^2
  \;\le\; m^3\,\frac{16}{\pi\varepsilon^2}\,\DD(u,u),
\]
where $\delta_eu$ is the difference of $u$ across $e$. Since
$\sum_{S}w(u-\bar u_S)^2=(2\vol_w(S))^{-1}\sum_{p,q\in S}w(p)w(q)(u(p)-u(q))^2$,
the claim follows from~\eqref{eq:thick-weight}.
\end{proof}

\begin{lemma}
\label{lem:neck}
For every $u:V\to\R$,
\[
  \sum_{T^\pm}w\,\bigl(u-u(q_\pm)\bigr)^2\;\le\;\alpha_4\,\varepsilon^{-2}\,
  \DD(u,u),
  \qquad
  \alpha_4=\frac{16\,\alpha_2e^{8\varepsilon}}{\alpha_1}.
\]
\end{lemma}

\begin{proof}
Order $T^\pm$ by $|\rho|$; consecutive vertices are adjacent, with
$|\rho|$-gaps in $[\varepsilon/2,2\varepsilon]$, and their node and edge
weights $w_j,\mu_j$ satisfy~\eqref{eq:collar-weight}. The one-sided
weighted Hardy inequality~\cite[Proposition~1]{Miclo} for $u$ vanishing
at $q_\pm$ gives
$\sum_jw_ju_j^2\le4H_\pm\sum_j\mu_j(u_{j+1}-u_j)^2$ with
$H_\pm=\sup_k\bigl(\sum_{|\rho_j|\le|\rho_k|}w_j\bigr)
\bigl(\sum_{|\rho_k|\le|\rho_j|<\rho_c}\mu_j^{-1}\bigr)$, both sums over
$T^\pm$. Comparing sums with integrals by monotonicity of $\cosh$ and
$\operatorname{sech}$,
\[
  H_\pm\;\le\;\frac{4\alpha_2e^{4\varepsilon}}{\alpha_1\varepsilon^{2}}
  \sup_{\rho^*}\;
  \sinh(\rho^*+2\varepsilon)\!\int_{\rho^*-2\varepsilon}^{\infty}
  \!\operatorname{sech}
  \;\le\;\frac{4\,\alpha_2e^{8\varepsilon}}{\alpha_1\varepsilon^{2}},
\]
since $\int_r^\infty\operatorname{sech}=2\arctan e^{-r}\le2e^{-r}$ and
$\sinh r\le\tfrac12e^r$ give
$\sinh(\rho^*+2\varepsilon)\int_{\rho^*-2\varepsilon}^\infty\operatorname{sech}\le e^{4\varepsilon}$. Applying this to $u-u(q_\pm)$ proves the lemma.
\end{proof}

\subsection{Proof of Proposition~\ref{prop:intro-swy}}

Let $\langle f,\mathbf1\rangle_w=0$, $\normw{f}^2=1$, $Q=\DD(f,f)$, and
$a_\pm=\bar f_{S^\pm}$. By Lemma~\ref{lem:side-poincare},
\begin{equation}
\label{eq:side}
  \sum_{S^\pm}w\,(f-a_\pm)^2\;\le\;\alpha_3\,Q .
\end{equation}
Since $w(p_\pm)\geq\pi\varepsilon^2/16$, \eqref{eq:side} gives
$(f(p_\pm)-a_\pm)^2\leq16\alpha_3Q/(\pi\varepsilon^2)$. Since
$\pm\rho(q_\pm)\geq\rho_c-2\varepsilon$, \eqref{eq:collar-weight} gives
$w(q_\pm)\geq\alpha_1\varepsilon e^{-2\varepsilon}\geq\pi\varepsilon^2/16$,
so $\mu(p_\pm,q_\pm)\geq\pi\varepsilon^2/16$ and
$(f(q_\pm)-f(p_\pm))^2\leq16Q/(\pi\varepsilon^2)$. Hence
\begin{equation}
\label{eq:endpoint}
  \bigl(f(q_\pm)-a_\pm\bigr)^2\;\le\;\alpha_5\,\frac{Q}{\varepsilon^2},
  \qquad \alpha_5=\frac{32(\alpha_3+1)}{\pi}.
\end{equation}
By Lemma~\ref{lem:neck}, \eqref{eq:endpoint} and~\eqref{eq:vol-upper},
\begin{equation}
\label{eq:half-column}
  \sum_{T^\pm}w\,(f-a_\pm)^2\;\le\;\alpha_6\,\frac{Q}{\varepsilon^2},
  \qquad \alpha_6=2\alpha_4+8\pi N\alpha_5 .
\end{equation}

\emph{Separating case.} Let $g=a_\pm$ on $S^\pm\cup T^\pm$.
By~\eqref{eq:side} and~\eqref{eq:half-column},
$\sum_Vw(f-g)^2\leq2(\alpha_3+\alpha_6)Q/\varepsilon^2$. Since
$w(c_\pm)\geq\alpha_1\ell$, Lemma~\ref{lem:neck} gives
$(f(c_\pm)-f(q_\pm))^2\leq\alpha_4Q/(\alpha_1\varepsilon^2\ell)$, and
since $\mu(c_+,c_-)\geq\alpha_1\ell$,
$(f(c_+)-f(c_-))^2\leq Q/(\alpha_1\ell)$. Together
with~\eqref{eq:endpoint}, $\ell\leq1$ and $\varepsilon\leq1$,
\[
  (a_+-a_-)^2\;\le\;\alpha_7\,\frac{Q}{\varepsilon^2\ell},
  \qquad
  \alpha_7=5\Bigl(2\alpha_5+\frac{2\alpha_4+1}{\alpha_1}\Bigr).
\]
Since $\langle f,\mathbf1\rangle_w=0$, using~\eqref{eq:vol-upper},
\[
  1\;\le\;\normw{f-a_+}^2
  \;\le\;2\sum_Vw\,(f-g)^2+2\vol_w(S^-\cup T^-)(a_+-a_-)^2
  \;\le\;\bigl(4(\alpha_3+\alpha_6)+8\pi N\alpha_7\bigr)\frac{Q}{\varepsilon^2\ell},
\]
so $\lambda_1(X_\ell)\geq C_1\ell$ with
$C_1=\varepsilon^2/\bigl(4(\alpha_3+\alpha_6)+8\pi N\alpha_7\bigr)$.

For the upper bound, let $\tilde h=\rho$ on vertices with $|\rho|\leq1$,
and $\tilde h=\pm1$ on the remaining vertices of $\Sigma^\pm$. Put
$h=\tilde h-\beta\mathbf1$ with $\langle h,\mathbf1\rangle_w=0$.
By~\eqref{eq:vol-lower},
\[
  \normw{h}^2\;\geq\;\vol_w(S^+)(1-\beta)^2+\vol_w(S^-)(1+\beta)^2
  \;\geq\;\frac{4\vol_w(S^+)\vol_w(S^-)}{\vol_w(S^+)+\vol_w(S^-)}
  \;\geq\;2\pi.
\]
If $\tilde h(p)\neq\tilde h(q)$ for an edge $\{p,q\}$, then both
endpoints satisfy $|\rho|<1+3\varepsilon<\rho_c$, so
$|\tilde h(p)-\tilde h(q)|<3\varepsilon$ and
$\mu(p,q)\leq\alpha_2\ell\cosh(1+3\varepsilon)$
by~\eqref{eq:collar-weight}. There are at most $5(1+3\varepsilon)/\varepsilon$
vertices with $|\rho|<1+3\varepsilon$, each with at most $N$ neighbours,
so $\DD(h,h)\leq45N\alpha_2(1+3\varepsilon)\cosh(1+3\varepsilon)\,
\varepsilon\ell$, and the min--max principle gives
\[
  \lambda_1(X_\ell)\le C_2\,\ell,
  \qquad
  C_2=\frac{45N\alpha_2(1+3\varepsilon)\cosh(1+3\varepsilon)\,\varepsilon}{2\pi}.
\]

\emph{Non-separating case.} Here $a_+=a_-=:a$ and
$V=S^+\cup T^+\cup T^-$. Since $\langle f,\mathbf1\rangle_w=0$,
\eqref{eq:side} and~\eqref{eq:half-column} give
\[
  1\;\le\;\sum_Vw\,(f-a)^2\;\le\;(\alpha_3+2\alpha_6)\,\frac{Q}{\varepsilon^2},
\]
so $\lambda_1(X_\ell)\geq C_3$ with $C_3=\varepsilon^2/(\alpha_3+2\alpha_6)$.
\qed

\begin{proof}[Proof of Corollary~\ref{cor:intro-swy}]
Let $c_1,c_2$ be the constants of Theorem~\ref{thm:main} with $n=2$ and
$\kappa=1$. Applying that theorem to $X_\ell$ and combining with
Proposition~\ref{prop:intro-swy} gives the corollary with
$\tilde C_1=c_1C_1$, $\tilde C_2=c_2C_2$ and $\tilde C_3=c_1C_3$.
\end{proof}

\printbibliography

\end{document}